\documentclass[12pt]{article}   	
\usepackage{geometry}                		
\usepackage[parfill]{parskip}    		
\usepackage{graphicx}				
\usepackage[pdftex,pagebackref,letterpaper=true,colorlinks=true,pdfpagemode=none,urlcolor=blue,linkcolor=blue,citecolor=blue,pdfstartview=FitH]{hyperref}
\usepackage{amsmath,amsfonts}
\usepackage{color}
\usepackage{amsthm}
\usepackage{boolexpr}
\usepackage{amssymb,latexsym,pstricks,pst-plot}
\usepackage[english]{babel}
\usepackage{pgf}
\usepackage{tikz}
\usetikzlibrary{arrows,automata}
\usepackage[latin1]{inputenc}
\usepackage[UKenglish]{isodate}
\allowdisplaybreaks

\usepackage{amssymb}
\newtheorem{thm}{Theorem}

\title{Structure and asymptotics for Motzkin numbers modulo primes using automata}
\author{Rob Burns}

\begin{document}
\maketitle
\begin{abstract}
We establish a lower bound of $\frac{2}{p (\, p - 1 \,)}$ for the asymptotic density of the Motzkin numbers divisible by a general prime number $p \geq 5$. We provide a criteria for when this asymptotic density is actually $1$. We also provide a partial characterisation of those Motzkin numbers which are divisible by a prime $p \geq 5$. All results are obtained using the automata method of Rowland and Yassawi.
\end{abstract}

\section{Introduction}
The Motzkin numbers $M_n$ are defined by
$$
M_n := \sum_{k \geq 0} \binom {n}{2k} C_{k}
$$
where $C_{k}\, $ are the Catalan numbers. 

There has been some work in recent years on analysing the Motzkin numbers $M_n$ modulo primes and prime powers. Deutsch and Sagan \cite{Sagan2006} provided a characterisation of Motzkin numbers divisible by $2, 4$ and $5$. They also provided a complete characterisation of the Motzkin numbers modulo $3$ and showed that no Motzkin number is divisible by $8$. Eu, Liu and Yeh \cite{Eu2008} reproved some of these results and extended them to include criteria for when $M_n$ is congruent to $\{ 2, 4, 6 \} \mod 8$. Krattenthaler and M{\"u}ller \cite{KM2013} established identities for the Motzkin numbers modulo higher powers of 3 which include the modulo 3 result of  \cite{Sagan2006} as a special case. Krattenthaler and M{\"u}ller \cite{Krat2016} have more recently extended this work to a full characterisation of \mbox{$M_n \mod 8$} in terms of the binary expansion of $n$. The results in \cite{KM2013} and \cite{Krat2016} are obtained by expressing the generating function of $M_n$ as a polynomial involving a special function. Rowland and Yassawi \cite{RY2013} investigated $M_n$ in the general setting of automatic sequences.  The values of $M_n$ (as well as other sequences) modulo prime powers can be computed via automata. Rowland and Yassawi provided algorithms for creating the relevant automata. They established results for $M_n$ modulo small prime powers, including a full characterisation of $M_n$ modulo 8 (modulo $5^2$ and $13^2$ are available from Rowland's website). They also established that $0$ is a forbidden residue for $M_n$ modulo $8$, $5^2$ and $13^2$. In theory the automata can be constructed for any prime power but computing power and memory quickly becomes a barrier. For example, the automata for $M_n$ modulo $13^2$ has over 2000 states. Rowland and Yassawi also went on to describe a method for obtaining asymptotic densities of $M_n$. We have previously \cite{1612.08146} used Rowland and Yassawi's work to analyse Motzkin numbers modulo specific primes up to $29$. This current paper will deal with a general prime $p$. It turns out that the behaviour of $M_n$ modulo a general prime is similar to the behaviour modulo small primes.

We will use Rowland and Yassawi's automata to establish a lower bound on the asymptotic density of the set of $M_n$ divisible by a general prime $p \geq 5$. This lower bound is $\frac{2}{p(p-1)}$. As shown in \cite{1612.08146} the asymptotic density is actually $1$ for some primes, e.g. $p = 7, 17,19$. We will also make note of some structure results that appear from an examination of the relevant state diagrams of the automata. In particular $M_n \equiv 0 \mod p$ when $n$ takes certain forms depending on the prime $p$. This generalises results that had already been shown to hold for particular small primes as mentioned in the previous paragraph. It is found that the behaviour of $M_n \mod p$ depends to some extent on the value of $p~\mod~6$ - this is either $+1$ or $-1$ of course. 

Table~\ref{results} summarises the results that will be presented in subsequent sections. Firstly, we will explain some of the definitions we have used in this paper.

The asymptotic density of a subset $S$ of $\mathbb{N}$ is defined to be
$$
\lim_{N \to \infty} \frac {1}{N} \#\{ n \in S : n \leq N \}
$$
if the limit exists, where $\#S$ is the number of elements in a set $S$. 
For a prime number $p \geq 5$ we mainly will be studying the asymptotic density of the set
\begin{equation}
\mbox{$S_p(0) = \{ n \in \mathbb{N}: M_n \equiv 0 \mod p \}$ .}
\end{equation}
However, for any $x \in \mathbb{N}$ we define the set 
\begin{equation*}
\mbox{$S_p(x) = \{ n \in \mathbb{N}: M_n \equiv x \mod p \}$ .}
\end{equation*}
\bigskip

For a number $p$, we write the base $p$ expansion of a number $n$ as 
$$
[\, n ]\,_{p} = \langle n_{r} n_{r-1} ... n_{1} n_{0} \rangle
$$ 
where $n_{i} \in [\, 0, p - 1]\,$  and 
$$
n = n_{r} p^{r} + n_{r-1} p^{r-1} + ... + n_{1} p + n_{0}.
$$ 

Binomial coefficients are prominent in this paper. Here the binomial coefficient $\binom{n}{m}$ is defined to be $0$ when $m > n$ or when either $n$ or $m$ is negative.

\bigskip

\begin{table}[tbp]
  \centering
   \begin{tabular}{ | c | c | p{10cm} |}
 \hline
 Prime & Density & Values of $n$ such that $M_n \equiv 0 \mod p$ \\ \hline
 & & \\
$p \equiv 1 \mod 6$ & $\geq \frac{2}{p(p-1)}$ & \mbox{$n = (pi + 1)p^k - 2 \;$ for $\; i \geq 0 \;$ and $\; k \geq 1$.} \\
& & \mbox{$n = (pi + p - 1)p^k - 1 \;$ for $\; i \geq 0\; $ and $\; k \geq 1$.} \\ 
& & \\ \hline
& & \\
& & \mbox{$n = (pi + 1)p^{2k} - 2 \; $ for $\; i \geq 0 \;$ and $\; k \geq 1$.} \\
$p \equiv -1 \mod 6$ & $\geq \frac{2}{p(p-1)}$ & \mbox{$n = (pi + p - 2)p^{2k+1} - 2 \;$ for $\; i \geq 0 \;$ and $\; k \geq 0$.} \\
& & \mbox{$n = (pi + 2)p^{2k+1} - 1 \;$ for $\; i \geq 0 \;$ and $\; k \geq 0$.} \\
& & \mbox{$n = (pi + p - 1)p^{2k} - 1 \;$ for $\; i \geq 0 \;$ and $\; k \geq 1$.} \\ & & \\ \hline
  \end{tabular}
  \caption{Table of results}
  \label{results}
\end{table}

\bigskip

\section{Background on Motzkin numbers modulo primes}

As mentioned in the introduction there have been results which characterise $M_n$ modulo primes $p \leq 29$. We collect these below to allow comparison with the results for a general prime.

\begin{thm}
\label{mod2}
(Theorem 5.5 of \cite{Eu2008}). The $n$th Motzkin number $M_n$ is even if and only if 
$$
\mbox{$n = (4i + \epsilon)4^{j+1} - \delta$ for $i, j \in \mathbb{N}, \epsilon \in \{1, 3\}$ and $\delta \in \{1, 2\}$. }
$$ 
\end{thm}
\bigskip

\begin{thm}
\label{mod3}
(Corollary 4.10 of \cite{Sagan2006}). Let $T (\, 01 )\,$ be the set of numbers which have a base-$3$ representation consisting of the digits $0$ and $1$ only. Then the Motzkin numbers satisfy
$$
M_n \equiv
\begin{cases}
-1 \mod 3 & \quad \text{ if  $\quad n \in \, 3T (\, 01 )\, - 1$,} \\
1 \mod 3  & \quad \text{if  $\quad n \in 3T(\, 01 ) \quad$ or $\quad n \in 3T(\, 01 )\, - 2$,} \\
0 \mod 3 & \quad \text{otherwise.}
\end{cases}
$$
\end{thm}

\bigskip

\begin{thm}
\label{mod5}
(Theorem 5.4 of \cite{Sagan2006}). The Motzkin number $M_n$ is divisible by $5$ if and only if $n$ is one of the following forms
$$
(\, 5i + 1 )\, 5^{2j}  - 2,\, (\, 5i + 2 )\, 5^{2j-1}  - 1,\, (\, 5i + 3 )\, 5^{2j-1}  - 2,\,  (\, 5i + 4 )\, 5^{2j}  - 1
$$
where $i, j \in \mathbb{N}$ and $j \geq 1$.
\end{thm}

\bigskip
The above results and others have been used to establish asymptotic densities of the sets $S_q(0)$
for $q = 2, 4, 8$ and also for primes up to $29$ - see \cite{RY2013}, \cite{Burns:2016v0} and \cite{1612.08146}. In particular the asymptotic density of $S_2(0)$ is $\frac{1}{3}$ (\cite{RY2013} example 3.12), the asymptotic density of $S_4(0)$ is $\frac{1}{6}$ (\cite{RY2013} example 3.14), the asymptotic density of $S_q(0)$ is $1$ for $q \in \{3, 7, 17, 19 \}$ \cite{Burns:2016v0} and \cite{1612.08146}, the asymptotic density of $S_q(0)$ is $\frac{2}{q(q-1)}$ for $q \in \{ 5, 11, 13, 23 \}$ \cite{Burns:2016v0} and \cite{1612.08146}. The asymptotic density of $S_{29}(0)$ satisfies \mbox{$\frac{2}{29*28} < S_{29}(0) < 1$}. 

There are 2 questions which we will investigate in this article. Firstly, what is the asymptotic density of $S_p(0)$ for general primes $p \geq 5$? Secondly, what structural features are evident in the distribution of $M_n \mod p$? The investigation will proceed by constructing an automaton for a general prime following the instructions from \cite{RY2013} (Algorithm 1). The state diagram for the automaton provides an excellent tool for analysing the behaviour of $M_n \mod p$. 

\bigskip

\section{Three useful series and some modular identities involving binomial coefficients}

There are $3$ series that will appear regularly during the construction of the automaton. We will therefore devote this section to a discussion of these series which are interesting in their own right. We define the series as follows:

\bigskip

\begin{equation}
\label{an}
\mbox{$a_n := \sum_{k \geq 0} (-1)^k \, \binom{n-k}{k}$}
\end{equation}
\bigskip
\begin{equation}
\label{bn}
\mbox{$b_n := \sum_{k \geq 0} (-1)^k \, \binom{n-k}{k} \, k$}
\end{equation}
\bigskip
\begin{equation}
\label{cn}
\mbox{$c_n := \sum_{k \geq 0} (-1)^k \, \binom{n-k}{k} \, k^2$.}
\end{equation}

\bigskip

\begin{thm}
\label{abcseries}
$$
a_n = \cos(\, \frac{\pi n}{3} \,) + \frac{1}{\sqrt{3}} \sin(\,\frac{\pi n}{3} \,)
$$
$$
b_n = \frac{2 n}{3} \cos(\, \frac{\pi n}{3} \,) - \frac{2}{3 \sqrt{3}} \sin(\,\frac{\pi n}{3} \,)
$$ and 
$$
c_n = \frac{n}{3} (\, n - 1 \,) \cos(\, \frac{\pi n}{3} \,) - \frac{1}{3 \sqrt{3}} (\, n^2 + n - 2 \,) \sin(\,\frac{\pi n}{3} \,)
$$
\end{thm}
\begin{proof}
The three series each satisfy a linear difference equation. The solutions to these equations can be derived using standard methods. For $a_n$ we have $a_0 = 1$, $a_1 = 1$ and
\begin{eqnarray*}
a_{n+1} &=& \sum_{k \geq 0} (-1)^k \, \binom{n+1-k}{k}\\
&=& \sum_{k \geq 0} (-1)^k \, \binom{n-k}{k} + \sum_{k \geq 0} (-1)^k \, \binom{n-k}{k-1}\\
&=& \sum_{k \geq 0} (-1)^k \, \binom{n-k}{k} - \sum_{k \geq 0} (-1)^k \, \binom{n-1-k}{k}
\end{eqnarray*}
where we have used the binomial identity
\begin{equation}
\label{bin1}
\binom{s+1}{r} = \binom{s}{r} + \binom{s}{r-1}.
\end{equation}
So $a_n$ satisfies the difference equation
\begin{equation}
\label{adiff}
a_{n+1} - a_n + a_{n-1} = 0
\end{equation}
The solution of this difference equation is given in the statement of the theorem. The initial values of $b_n$ are $b_0=0$, $b_1=0$ and, using the identity (\ref{bin1}) again, we can show $b_n$ satisfies the non-homogeneous difference equation
\begin{equation}
\label{bdiff}
b_{n+1} - b_n + b_{n-1} = -a_{n-1}.
\end{equation}
Finally, $c_0=c_1=0$ and $c_n$ satisfies the non-homogeneous difference equation
\begin{equation}
\label{cdiff}
c_{n+1} - c_n + c_{n-1} = -a_{n-1} - 2 b_{n-1}.
\end{equation}
\end{proof}
\bigskip
The results can be extended to the series
$$
\mbox{$\sum_{k \geq 0} (-1)^k \, \binom{n-k}{k} \, k^m$.}
$$
where $m$ is arbitrary but we only need the equations for $a_n$,  $b_n$ and $c_n$.
We list here some of the properties of the series which will be needed later.
Firstly, $a_n$ is a periodic sequence with period $6$. Starting with $n = 0$ the sequence $a_n$ is \mbox{$1, 1, 0, -1, -1, 0, 1, 1, ...$}. For any $n \in \mathbb{N}$, 

$$
a_{n-1} =
\begin{cases}
\label{a_n-1_cases}
1  \; & \text{if $\; n \equiv 1 \mod 6$;}\\
-1 \; & \text{if $\; n \equiv -1 \mod 6$;}\\
\end{cases}
$$

$$
a_{n} =
\begin{cases}
\label{a_n_cases}
1 \; & \text{if $\; n \equiv 1 \mod 6$;}\\
0 \; & \text{if $\; n \equiv -1 \mod 6$;}\\
\end{cases}
$$

$$
b_{n} =
\begin{cases}
\label{b_n_cases}
\frac{n-1}{3}  \; & \text{if $\; n \equiv 1 \mod 6$;}\\
\frac{n+1}{3}  \; & \text{if $\; n \equiv -1 \mod 6$;}\\
\end{cases}
$$
and
$$
c_{n} =
\begin{cases}
\label{a_p_cases}
-\frac{(n-1)}{3}  \; & \text{if $\; n \equiv 1 \mod 6$;}\\
\frac{n^2 - 1}{3}  \; & \text{if $\; n \equiv -1 \mod 6.$}\\
\end{cases}
$$

The above identities for $a_n$ show that if $p$ is prime (so $p \equiv \frac{+}{} 1 \mod 6$)
\bigskip
\begin{equation}
\label{anid}
a_{p-1} - 2 a_p + 1 = 0
\end{equation}

\bigskip
and
\bigskip
\begin{equation}
\label{bcnid}
b_p + c_p =
\begin{cases}
0  \; & \text{if $\; p \equiv 1 \mod 6$;}\\
\frac{p(\, p + 1 \,)}{3}  \; & \text{if $\; p \equiv -1 \mod 6.$}\\
\end{cases}
\end{equation}

\bigskip
We will also need the following identities
\begin{equation}
\label{long}
 a_p - b_p - 1 + \frac{1}{2} b_{p+1} + \frac{1}{2} c_{p+1} - b_{p+2} - c_{p+2} =
 \begin{cases}
 \frac{p (p + 5)}{6}  \; & \text{if $p \equiv 1 \mod 6$;}\\
\frac{p(p + 1)}{6} - 1  \; & \text{if $p \equiv -1 \mod 6.$}\\
\end{cases}
 \end{equation}

\bigskip
\begin{equation}
\label{long2}
 - a_{p-1} + a_p + b_p - 2 b_{p+1} - 2 a_{p+1} + 1 =
 \begin{cases}
 p + 2 \; & \text{if $p \equiv 1 \mod 6$;}\\
-p - 1  \; & \text{if $p \equiv -1 \mod 6.$}\\
\end{cases}
 \end{equation}
 \bigskip
 
There are also a few modular identities that will be useful in simplifying some of the equations that will appear later. Firstly, for $k, l \in \mathbb{N}$,
\begin{equation}
\label{p-1-k}
\binom{p-1-k}{l} \equiv (-1)^l \; \binom{k+l}{k} \mod p.
\end{equation}
In particular,
\begin{eqnarray*}
\binom{p-1}{l} &\equiv& (-1)^l \mod p\\
\binom{p-2}{l} &\equiv& (-1)^l (\, l + 1 \,) \mod p\\
\binom{p-3}{l} &\equiv& (-1)^l \, \binom{l+2}{2} \equiv  \frac{1}{2} \, (-1)^l \, (\, l + 2 \,) (\, l + 1 \,) \mod p.\\
& &
\end{eqnarray*}

\bigskip

\section{Background on automata for $M_n \mod p$}
Rowland and Yassawi showed in \cite{RY2013} that the behaviour of sequences such as $M_n~\mod~p$ can be studied by the use of finite state automata. The automaton has a finite number of states and rules for transitioning from one state to another. In the form described in \cite{RY2013} each state $s$ is represented by a polynomial in 2 variables $x$ and $y$. Each state has a value obtained by evaluating the polynomial at $x = 0$ and $y = 0$. All calculations are made modulo $p$. For the Motzkin case the initial state
$s_1$ is represented by the polynomial
\begin{equation}
\label{R}
R(\, x, y \, )\,  = \, y(\, 1 - xy - 2x^2y^2 - 2x^2y^3 \, ).
\end{equation}

New states are constructed by applying the Cartier operator $\Lambda_{d, d}$ to the polynomials
$$
s_i*Q^{p-1}(\, x, y \, )
$$
for $d \in \{0, 1, ... , p-1 \}$ where $\{ s_i \}$ are the already calculated states and the polynomial $Q$ is defined by
\begin{equation}
\label{Q}
Q(\, x, y \, )\,  = \, x^2y^3 + 2x^2y^2 + x^2y + xy + x - 1 = x^2y(y + 1)^2 + x(y + 1) - 1.
\end{equation}
The Cartier operator is a linear map on polynomials defined by
\begin{equation}
\label{cartier}
\Lambda_{d_1, d_2} (\, \sum_{m, n \geq 0} a_{m, n} \, x^{m} \, y^{n} \, ) = \sum_{m, n \geq 0} a_{pm + d_1, pn + d_2} \, x^{m} y^{n}.
\end{equation}

Since the Cartier operator maintains or reduces the degree of the polynomial and there are only finitely many polynomials modulo $p$ of each degree, all states of the automaton are obtained within a known finite time. It will be seen later that the automaton has at most $p+6$ states. If
$$
\Lambda_{d, d} (s*Q^{p-1}) = t
$$
for states (i.e. polynomials) $s$ and $t$ then the transition from state $s$ to state $t$ under the input $d$ is part of the automaton.

To calculate $M_n \mod p$, $n$ is first represented in base $p$. The base $p$ digits of $n$ are fed into the automaton starting with the least significant digit. The automaton starts at the initial state $s_1$ and transitions to a new state as each digit is fed into it. The value of the final state after all $n$'s digits have been used is equal to $M_n \mod p$. Refer to \cite{RY2013} for more details. 

In the remainder of this article we will provide details of the automata for a general prime $p \geq 5$. We will provide the polynomials and values for the states and the relevant transitions between states. States are listed as $s_1$, $s_2$, ... . Transitions, when provided, will be in the form $(\, s, j)\,  \to t$ which means that if the automaton is in state $s$ and receives digit $j$ then it will move to state $t$.  We will call a state $s$ a {\em \bf loop} state if all transitions from $s$ go to $s$ itself, i.e. \mbox{$(\, s, j)\,  \to s$} for all choices of $j$.

States and transitions are represented visually in the form of a directed graph. For example, figure \ref{transition} represents an automaton which moves from state $s_1$ to state $s_2$ when it receives the digit $3$. It also moves from state $s_2$ to state $s_2$ (i.e. loops) if it is in state $s_2$ and receives a digit $4$.

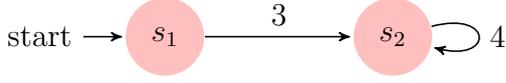
\begin{figure}[tbp]
\begin{tikzpicture}[->,>=stealth',shorten >=1pt,auto,node distance=3.0cm,
                    semithick]
  \tikzstyle{every state}=[fill=pink,draw=none,text=black]

  \node[initial,state] (A)                    {$s_1$};
  \node[state]         (B) [right of=A] {$s_2$};

  \path (A) edge    node {3} (B)
        (B) edge [loop right] node {4} (B);                      
 \end{tikzpicture}
                       \caption{Example of transition from state $s_1$ to state $s_2$ and a loop.}
                        \label{transition}
\end{figure}

\bigskip

\section{Preliminary calculations}

Before we start constructing the automata it will be convenient to first precompute $\Lambda_{d, d} (\, s(x, y)*Q(\, x, y \, )^{p-1} \,)$ for some simple choices of the polynomial $s$. The relevant results are contained in tables~\ref{lambdaactions-1} and \ref{lambdaactions-2}. When reading the table note that \mbox{$\binom{n}{m} = 0$} for \mbox{$m < 0$}. We will go through a few of the calculations from tables~\ref{lambdaactions-1} and \ref{lambdaactions-2}.

\begin{table}[tbp]
  \centering
   \begin{tabular}{ | c ||  p{10cm} |}
\hline
& \\
State $s$ &  $\Lambda_{d, d} (\, s*Q^{p-1} \,)$ \\
\hline \hline
& \\
& $1 \;$ for $\; \mathbf{0 \leq d \leq 1}$ \\
& \\
$1$ & $\sum_{k \geq 0} \binom{p-1-k}{d-2k} \binom{d}{k} (-1)^d \;$ for $\; \mathbf{2 \leq d \leq p - 3}$ \\
& \\
& $b_p  \;$ for $\mathbf{d = p-2}$ \\
& \\
& $a_{p-1} \;$ for $\mathbf{d = p-1}$ \\
\hline
& \\
& $0$ for $\mathbf{d = 0}$ \\
& \\
& $1$ for $\mathbf{d = 1}$ \\
&\\
& $\sum_{k \geq 0} \binom{p-1-k}{d-2k} \binom{d}{k+1} (-1)^d \;$ for $\; \mathbf{2 \leq d \leq p - 3}$ \\
$y$ & \\
& $-\sum_{k \geq 0} \binom{p-1-k}{k+1} \binom{p-2}{k+1} + xy + xy^2 \;$ \\
& \\
& $= -a_p - b_p + 1 + xy + xy^2 \;$ for $\; \mathbf{d = p - 2}$ \\
& \\
& $\sum_{k \geq 0} \binom{p-1-k}{k} \binom{p-1}{k+1} \; = - a_{p-1}$ for $\; \mathbf{d = p - 1}$ \\
\hline
& \\
& $-\sum_{k \geq 0} \binom{p-1-k}{k+1} \binom{p-2}{k} \; xy \; = b_p \; xy$ for $\; \mathbf{d = 0}$ \\
& \\
$x^2y^2$ & $\sum_{k \geq 0} \binom{p-1-k}{k} \binom{p-1}{k} \; xy \; = a_{p-1} \; xy \;$ for $\; \mathbf{d = 1}$ \\
& \\
&  $\sum_{k \geq 0} \binom{p-1-k}{d-2-2k} \binom{d-2}{k} (-1)^{d} \;$ for $\; \mathbf{2 \leq d \leq p - 1}$ \\
\hline
& \\
& $\sum_{k \geq 0} \binom{p-1-k}{k} \binom{p-1}{k+1} \; xy \; = -a_{p-1} \; xy \;$ for $\; \mathbf{d = 0}$ \\
& \\
& $0 \;$ for $\; \mathbf{d = 1}$ \\
$x y^2$ & \\
& $\sum_{k \geq 0} \binom{p-1-k}{d-1-2k} \binom{d-1}{k+1} (-1)^{d+1} \;$ for $\; \mathbf{2 \leq d \leq p-3}$ \\
& \\
& $-\frac{1}{2} (c_{p+1} + b_{p+1} ) \;$ for $\; \mathbf{d = p-2}$\\
&\\
& $-\sum_{k \geq 0} \binom{p-1-k}{k+1} \binom{p-2}{k+1} \; + xy + xy^2 \;$\\
& \\
& $= -a_p - b_p + 1 + xy + xy^2 \;$ for $\; \mathbf{d = p - 1}$ \\
\hline
  \end{tabular}
  \caption{Table of values of $\Lambda_{d, d} (\, s*Q^{p-1} \,)$}
  \label{lambdaactions-1}
\end{table}

\bigskip

\begin{table}[tbp]
  \centering
   \begin{tabular}{ | c ||  p{10cm} |}
 \hline
 & \\
State $s$ &  $\Lambda_{d, d} (\, s*Q^{p-1} \,)$ \\
& \\ \hline \hline
& \\
& $a_{p-1} \, x y \;$ for $\; \mathbf{d = 0}$ \\
&\\
& $1 \;$ for $\; \mathbf{d = 1}$\\
& \\
$xy$ & $\sum_{k \geq 0} \binom{p-1-k}{d-1-2k} \binom{d-1}{k} (-1)^{d+1} \;$ for $\; \mathbf{2 \leq d \leq p-3}$ \\
& \\
& $\frac{1}{2} (c_{p+1} - b_{p+1} ) \;$ for $\; \mathbf{d = p-2}$\\
& \\
& $b_{p} \;$ for $\; \mathbf{d = p-1}$\\
\hline
& \\
& $-\sum_{k \geq 0} \binom{p-1-k}{k+1} \binom{p-2}{k+1} \; xy + x^2 y^2 + x^2 y^3 \;$\\
& \\
& $= (-a_p - b_p + 1) \; xy + x^2 y^2 + x^2 y^3 \;$ for $\; \mathbf{d = 0}$ \\ 
& \\
$x^2y^3$ & $\sum_{k \geq 0} \binom{p-1-k}{k} \binom{p-1}{k+1} \; xy \; = -a_{p-1} \; xy \;$ for $\; \mathbf{d = 1}$ \\
& \\
& $0 \;$ for $\; \mathbf{d = 2}$ \\
& \\
&  $\sum_{k \geq 0} \binom{p-1-k}{d-2-2k} \binom{d-2}{k+1} (-1)^{d} \;$ for $\; \mathbf{3 \leq d \leq p - 2}$ \\
&\\
& $-\frac{1}{2} (c_{p+1} + b_{p+1} ) \;$ for $\; \mathbf{d = p-1}$\\
\hline
& \\
& $-\sum_{k \geq 0} \binom{p-1-k}{k+1} \binom{p-2}{k+2} \; xy -2 x^2 y^2 - 2 x^2 y^3 \;$\\
& \\
& $= (\, 2 a_p + b_p - 2 \,) x y -2 x^2 y^2 - 2 x^2 y^3 \;$ for $\; \mathbf{d = 0}$ \\
& \\
& $\sum_{k \geq 0} \binom{p-1-k}{k} \binom{p-1}{k+2} \; xy \; = a_{p-1} \; xy \;$ for $\; \mathbf{d = 1}$ \\
& \\
$x^2 y^4$ & $0 \;$ for $\; \mathbf{2 \leq d \leq 3}$ \\
& \\
& $\sum_{k \geq 0} \binom{p-1-k}{d-2-2k} \binom{d-2}{k+2} (-1)^d \;$ for $\; \mathbf{4 \leq d \leq p-3}$ \\
& \\
& $-\sum_{k \geq 0} \binom{p-1-k}{k+3} \binom{p-4}{k+2} + xy + xy^2 \;$ for $\; \mathbf{d = p-2}$ \\
& \\
&  $\sum_{k \geq 0} \binom{p-1-k}{k+2} \binom{p-3}{k+2} -xy-xy^2$\\
& \\
& $= \frac{1}{2} (c_{p+1} + 3 b_{p+1} + 2 a_{p+1} - 2 \;)-xy-xy^2 \;$ for $\; \mathbf{d = p - 1}$ \\
\hline
  \end{tabular}
  \caption{Table of values of $\Lambda_{d, d} (\, s*Q^{p-1} \,)$}
  \label{lambdaactions-2}
\end{table}
\bigskip

Firstly, the polynomial $Q^{p-1}$ can be written as
\begin{eqnarray*}
Q^{p-1} (x, y) &=& \biggr( \, x^2y(y + 1)^2 + x(y + 1) - 1 \, \biggr)^{p-1}\\
&=& \sum_{k \geq 0} \binom{p-1}{k} x^{2k} y^k (\, y+1 \,)^{2k} \biggr( \, x(y + 1) - 1 \, \biggr)^{p-1-k}\\
&=& \sum_{k, \, l  \geq 0} \binom{p-1}{k} x^{2k} y^k (\, y+1 \,)^{2k} \binom{p-1-k}{l} x^l (y + 1 )^l ( - 1 )^{p-1-k-l}\\
&=& \sum_{k,\,  l  \geq 0} \binom{p-1}{k} \binom{p-1-k}{l}  x^{2k+l} y^k (\, y+1 \,)^{2k+l} x^l ( - 1 )^{k + l}.\\
& &
\end{eqnarray*}
Using the identity
$$
\binom{p-1}{k} \equiv (-1)^k \mod p.
$$
\bigskip
we then have
\begin{equation}
\label{Qp-1}
Q^{p-1} (x, y) = \sum_{k, l, m \geq 0} \binom{p-1-k}{l} \binom{2k+l}{m} (\, -1 \,)^l x^{2k+l} y^{k+m}.
\end{equation}

\bigskip

We define $a_{k, l, m}$ by
\begin{equation}
\label{aklm}
a_{k, l, m} := \binom{p-1-k}{l} \binom{2k+l}{m} (\, -1 \,)^l.
\end{equation}
We then have
$$
Q^{p-1} (x, y) = \sum_{k, l, m \geq 0} a_{k, l, m} \, x^{2k+l} \, y^{k+m} = \sum_{i, j \geq 0} b_{i, j} \, x^i \, y^j
$$
where
$$
b_{i, j} = \;\;  \sum a_{k, l, m}
$$

and the sum is over $k, l, m \geq 0 : \; 2k + l = i, \; k + m = j$. So,

\begin{equation}
\label{bij}
b_{i, j}= \sum_{k \geq 0}  \binom{p-1-k}{ i - 2k} \binom{i}{j - k} (\, -1 \,)^{i}.
\end{equation}
\bigskip

We will next calculate the effect of the Cartier operator $\Lambda_{d, d}$ on a general monomial $x^r y^t$. For $0 \leq d \leq p-1$ we have
$$
\Lambda_{d, d} (x^r y^t Q^{p-1} ) = \Lambda_{d, d} (\, \sum_{i, j \geq 0} b_{i, j} x^{i + r} y^{j + t} \,). 
$$
$$
= \Lambda_{d, d} (\, \sum_{i \geq r, j \geq t} c_{i, j} x^{i} y^{j} \,) = \sum_{i \geq r, \; j \geq t} c_{pi + d, pj +d} x^i y^j
$$
where $c_{i, j}$ is defined by $c_{i,j} := b_{i-r, j-t}$. So
\begin{equation}
\label{cpij}
\Lambda_{d, d} (x^r y^t Q^{p-1} ) = \sum_{i \geq r, \; j \geq t} \biggr( \, \sum_{k \geq 0} \binom{p-1-k}{pi+d-r-2k} \; \binom{pi+d-r}{pj+d-t-k} \; (-1)^{i+d+r} \, \biggr) x^i \, y^j.
\end{equation}

\bigskip
The sum above is finite as the indices $i$ and $j$ satisfy the restrictions
$$
\mbox{$r \leq pi + d \leq 2(p-1) + r \;$ and $\; t \leq pj + d \leq 3(p-1) + t$.}
$$
We will first look at the monomial $1$, for which we have $r = t = 0$. For this choice of $r$ and $t$ we need to determine
$$
\binom{p-1-k}{pi+d-2k} \; \binom{pi+d}{pj+d-k} \; (-1)^{i+d}.
$$
for all possible choices of $i$ and $j$. In this case it turns out all terms are $0 \mod p$ except for the $i = j = 0$ term. So, for $0 \leq d \leq p-1$,
$$
\Lambda_{d, d} (\, 1*Q^{p-1} \,) = \sum_{k \geq 0} \binom{p-1-k}{d-2k} \binom{d}{k} (-1)^d. 
$$
\bigskip
For $d= p-2$ this sum reduces to $b_p$ as
\begin{eqnarray*}
\sum_{k \geq 0} \binom{p-1-k}{p-2-2k} \binom{p-2}{k} (-1)^{p-2} &=& \sum_{k \geq 0} \binom{p-1-k}{k+1} (-1)^{k+1} (k+1)\\
& &\\
&=& \sum_{k \geq 1} \binom{p-k}{k} (-1)^{k} k\\
& &\\
&=& b_p.\\
& &
\end{eqnarray*}
\bigskip
For $d = p-1$ the sum is $a_{p-1}$ as
$$
\sum_{k \geq 0} \binom{p-1-k}{p-1-2k} \binom{p-1}{k} (-1)^{p-1} = \sum_{k \geq 0} \binom{p-1-k}{k} (-1)^{k} = a_{p-1}
$$
\bigskip

The remaining cases can be treated similarly giving the results stated in table~\ref{lambdaactions-1} and table~\ref{lambdaactions-2}.

\bigskip

\section{Constructing the automata for $M_n \mod p$}
In this section we will describe the states and transitions of the automata for $M_n \mod p$. These are summarised in table~\ref{results-1} for the case $p~\equiv~1~\mod~6$ and tables~\ref{results-2} and \ref{results-3} for the case $p~\equiv -1~\mod~p$. A '$c$' appearing in the tables represents a constant state (constant polynomial). The value of $c$ depends on $d$ and $p$. For given $d: 0 \leq d \leq p-1$ and state $s$ the tables give the state equal to
$$
\Lambda_{d, d} (s*Q^{p-1}).
$$
The transition $(\, s, d \,) \to \Lambda_{d, d} (s*Q^{p-1})$ is then part of the automaton.

\bigskip

The behaviour of the automaton depends on the value of $\; p \mod 6$. When $p~\equiv~1~\mod~6$ there are up to $p+4$ states consisting of the $p$ constant polynomials modulo $p$ and the $4$ polynomials
$$
s_1 = -2x^2y^3(y+1) - xy^2 + y, \; s_2 = x^2y^2(y+1) + xy, \; -xy(y+1), \; x^2 + xy^2 +2.
$$
When $p~\equiv~-1~\mod~6$ there are up to $p+6$ states consisting of the $p$ constant polynomials modulo $p$ and the $6$ polynomials
$$
s_1, \; s_2, \; -xy(y+1) - 1, \; xy(y+1) - 1, \; xy(y+1), \; xy(y+1) + 2.
$$
It is unclear whether all $p$ constant polynomials always appear as states.

\bigskip

\begin{table}[tbp]
  \centering
   \begin{tabular}{ | c || c | c | c | c | c |}
 \hline
 & & & &\\
 $d$ & $s_1$ & $s_2$ & $1$ & $-xy(y+1)$ & $xy(y+1)+2$ \\ 
 & & $2x^2y^2(y+1)+xy$ & & &\\
 & & & & &\\ \hline \hline
 & & & & &\\
$d = 0$ & $s_2$ & $s_2$ & $1$ & $0$ & $2$\\ 
& & & & &\\ \hline
& & & & &\\
$d = 1$ & $1$ & $1$ & $1$ & $c$ & $c$\\ 
& & & & &\\ \hline
& & & & &\\
$2 \leq d \leq p-3$ & $c$ & $c$ & $c$ & $c$ & $c$\\ 
& & & & &\\ \hline
& & & & &\\
$d=p-2$ & $-xy(y+1)$ & $c$ & $c$ & $c$ & $0$\\
& & & & &\\ \hline
& & & & &\\
$d=p-1$ & $xy(y+1) + 2$ & $c$ & $1$ & $-xy(y+1)$ & $xy(y+1)+2$\\
& & & & &\\ \hline
  \end{tabular}
  \caption{Table of states and transitions for $p \equiv 1 \mod 6$.}
  \label{results-1}
\end{table}

\bigskip

\begin{table}[tbp]
  \centering
   \begin{tabular}{ | c || c | c | c | c |}
 \hline
 & & & &\\
 $d$ & $s_1$ & $s_2$ & $1$ & $-xy(y+1)-1$\\ 
 & & $2x^2y^2(y+1)+xy$ & &\\
 & & & &\\ \hline \hline
 & & & &\\
$d = 0$ & $s_2$ & $s_2$ & $1$ & $-1$\\ 
& & & &\\ \hline
& & & &\\
$d = 1$ & $1$ & $1$ & $1$ & $c$\\ 
& & & &\\ \hline
& & & &\\
$2 \leq d \leq p-4$ & $c$ & $c$ & $c$ & $c$\\ 
& & & &\\ \hline
& & & &\\
$d = p-3$ & $c$ & $c$ & $c$ & $0$\\ 
& & & &\\ \hline
& & & &\\
$d=p-2$ & $-xy(y+1)-1$ & $c$ & $c$ & $c$\\
& & & &\\ \hline
& & & &\\
$d=p-1$ & $xy(y+1) - 1$ & $c$ & $-1$ & $-xy(y+1)$ \\
& & & &\\ \hline
  \end{tabular}
  \caption{Table of states and transitions for $p \equiv -1 \mod 6$.}
  \label{results-2}
\end{table}

\bigskip

\begin{table}[tbp]
  \centering
   \begin{tabular}{ | c || c | c | c |}
 \hline
 & & &\\
 $d$ & $xy(y+1)-1$ & $-xy(y+1)$ & $xy(y+1)+2$\\ 
 & & &\\
 & & &\\ \hline \hline
 & & &\\
$d = 0$ & $-1$ & $0$ & $2$\\ 
& & &\\ \hline
& & &\\
$1 \leq d \leq p-3$ & $c$ & $c$ & $c$\\ 
& & &\\ \hline
& & &\\
$d=p-2$ & $c$ & $c$ & $0$\\
& & &\\ \hline
& & &\\
$d=p-1$ & $xy(y+1) + 2$ & $-xy(y+1)-1$ & $xy(y+1)-1$\\
& & &\\ \hline
  \end{tabular}
  \caption{Table of states and transitions for $p \equiv -1 \mod 6$.}
  \label{results-3}
\end{table}

Figures~\ref{modpfigure1}, \ref{modpfigure2} and \ref{modpfigure3} provide an alternative pictorial summary of the automata for $M_n~\mod~p$.

\bigskip

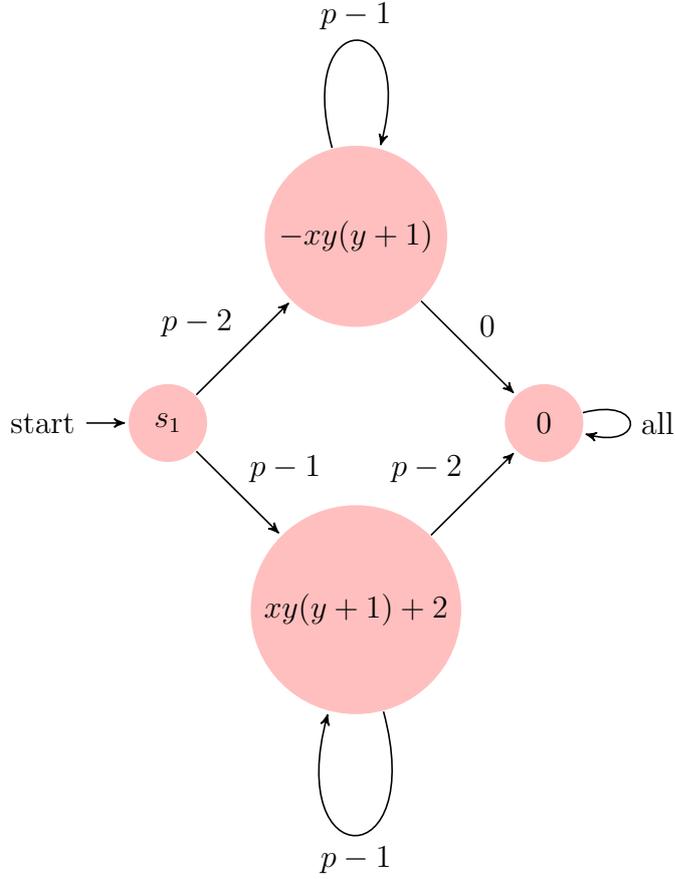
\begin{figure}[tbp]
\begin{tikzpicture}[->,>=stealth',shorten >=1pt,auto,node distance=3.5cm,
                    semithick]
  \tikzstyle{every state}=[fill=pink,draw=none,text=black]

  \node[initial,state] (A)                    {$s_1$};
  \node[state]         (B) [below right of=A] {$xy(y+1) + 2$};
  \node[state]         (C) [above right of=B] {$0$};
  \node[state]         (D) [above right of=A] {$-xy(y+1)$};

  \path (A) edge    node {$p-1$} (B)
            edge node {$p-2$} (D)
        (B) edge [loop below] node {$p-1$} (B)
        edge node {$p-2$} (C)
         (C) edge [loop right] node {all} (C)
                      (D) edge [loop above] node {$p-1$} (D)
                      edge node {$0$} (C);
                      
 \end{tikzpicture}

                       \caption{Partial state diagram for $M_n \mod p$ when $p \equiv 1 \mod p$.}
                        \label{modpfigure1}
\end{figure}

\bigskip

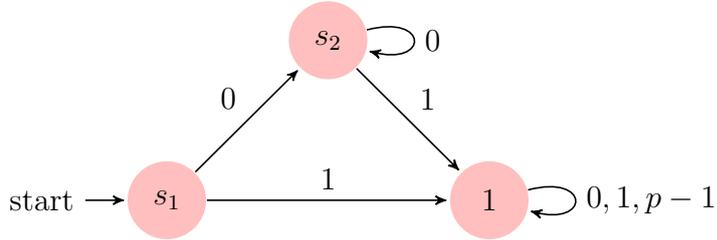
\begin{figure}[tbp]
\begin{tikzpicture}[->,>=stealth',shorten >=1pt,auto,node distance=3.0cm,
                    semithick]
  \tikzstyle{every state}=[fill=pink,draw=none,text=black]

  \node[initial,state] (A)                    {$s_1$};
    \node[state]         (C) [above right of=A] {$s_2$};
    \node[state]         (B) [below right of=C] {$1$};

  \path (A) edge    node {$1$} (B)
            edge node {$0$} (C)
        (B) edge [loop right] node {$0, 1, p-1$} (B)
         (C) edge [loop right] node {$0$} (C)
                     edge node {$1$} (B);
                      
 \end{tikzpicture}

                       \caption{Another part of the state diagram for $M_n \mod p$ when $p \equiv 1 \mod p$.}
                        \label{modpfigure2}
\end{figure}

\bigskip

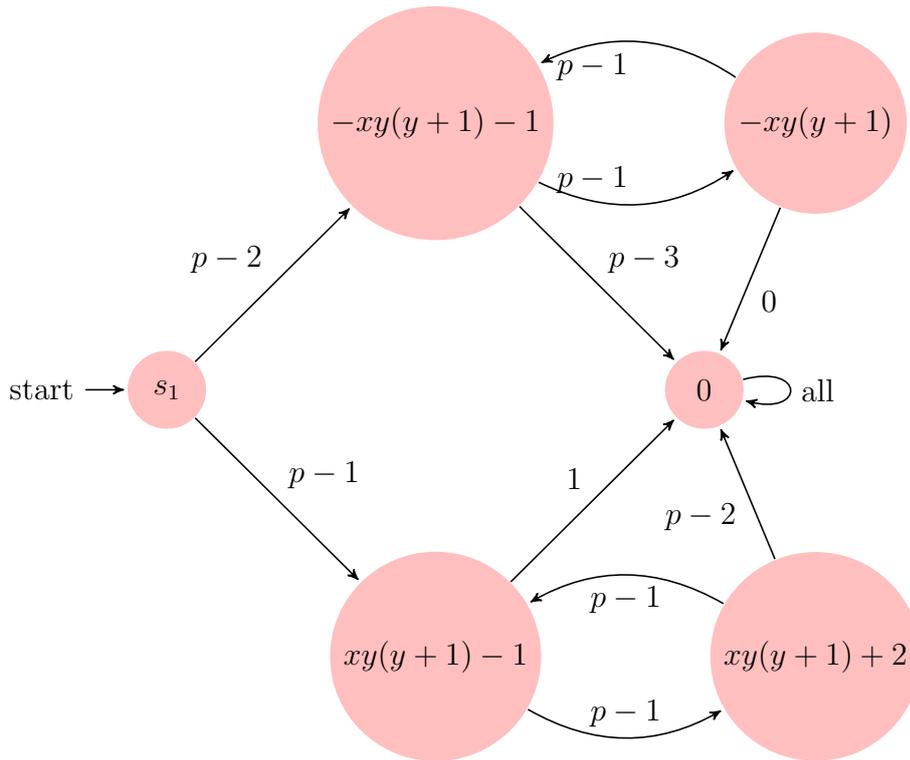
\begin{figure}[tbp]
\begin{tikzpicture}[->,>=stealth',shorten >=1pt,auto,node distance=5.0cm,
                    semithick]
  \tikzstyle{every state}=[fill=pink,draw=none,text=black]

  \node[initial,state] (A)                    {$s_1$};
  \node[state]         (B) [below right of=A] {$xy(y+1) - 1$};
  \node[state]         (E) [right of=B] {$xy(y+1) + 2$}; 
  \node[state]         (D) [above right of=A] {$-xy(y+1) - 1$};
   \node[state]         (F) [right of=D] {$-xy(y+1)$};
    \node[state]         (C) [below right of=D] {$0$};

  \path (A) edge    node {$p-1$} (B)
            edge node {$p-2$} (D)
        (B) edge [bend right] node {$p-1$} (E)
        edge node {$1$} (C)
         (C) edge [loop right] node {all} (C)
                      (D) edge [bend right] node {$p-1$} (F)
                      edge node {$p-3$} (C)
                      (E) edge [bend right] node {$p-1$} (B)
                      edge node {$p-2$} (C)
                       (F) edge [bend right] node {$p-1$} (D)
                      edge node {$0$} (C);
                      
 \end{tikzpicture}

                       \caption{Partial state diagram for $M_n \mod p$ when $p \equiv -1 \mod p$.}
                        \label{modpfigure3}
\end{figure}

\bigskip

\bigskip

The calculation of the states will rely on the data contained in table~\ref{lambdaactions-1} and table~\ref{lambdaactions-2}. As mentioned earlier, the initial state $s_1$ for the automata is the polynomial defined in equation~(\ref{R}). The second state $s_2$ is then given by
$$
s_2 = \Lambda_{0, 0} (\, s_1*Q(\, x, y \, )^{p-1} \,)
$$
$$
= \Lambda_{0, 0} (\, y(\, 1 - xy - 2x^2y^2 - 2x^2y^3 \, )*Q(\, x, y \, )^{p-1} \,)
$$
$$
= a_{p-1}  xy - 2 \biggr( (-a_p - b_p + 1) xy + x^2y^2 + x^2 y^3 \biggr) - 2 \biggr( ( 2 a_p + b_p -2 )xy - 2 x^2 y^2 - 2 x^2 y^3 \biggr)
$$
$$
= 2 x^2 y^2 + 2 x^2 y^3 + (\, a_{p-1} - 2 a_p + 2 \,) xy
$$
$$
= 2 x^2 y^2 + 2 x^2 y^3 + x y
$$
from equation~(\ref{anid}).

\bigskip

The next interesting state is $\Lambda_{p-2, p-2} (\, s_1*Q(\, x, y \, )^{p-1})$. We have

\begin{eqnarray*}
& &\Lambda_{p-2, p-2} (\, s_1*Q(\, x, y \, )^{p-1} \,) = -a_p - b_p + 1 +xy + xy^2 + \frac{1}{2}(b_{p+1} + c_{p+1}) \\
& & + \, 2(\sum_{k \geq 0} \binom{p-1-k}{k+3} \binom{p-4}{k+1} )  + 2(\sum_{k \geq 0} \binom{p-1-k}{k+3} \binom{p-4}{k+2} - xy - xy^2 )\\
&=& -xy(y+1) - a_p - b_p + 1 + \frac{1}{2} (b_{p+1} + c_{p+1} ) + 2(\sum_{k \geq 0} \binom{p-1-k}{k+3} \binom{p-3}{k+2} ).\\ 
& &
\end{eqnarray*}

Now since
\begin{eqnarray*}
&\sum_{k \geq 0}& \binom{p-1-k}{k+3} \binom{p-3}{k+2} = \frac{1}{2} \sum_{k \geq 0} \binom{p-1-k}{k+3} (-1)^k (k+4)(k+3) \\
&=& -\frac{1}{2} \sum_{k \geq 3} \binom{p+2-k}{k} (-1)^k k(k+1) \\ 
&=& -\frac{1}{2} ( c_{p+2} + (p+1) - 4 \binom{p}{2} + b_{p+2} + (p+1) - 2 \binom{p}{2} ) \\ 
&\equiv & - \frac{1}{2} (c_{p+2} + b_{p+2} + 2) \mod p \\
& &
\end{eqnarray*}
we have
$$
\Lambda_{p-2, p-2} (\, s_1*Q(\, x, y \, )^{p-1} \,) = -xy(y+1) - a_p - b_p - 1 + \frac{1}{2} b_{p+1} + \frac{1}{2} c_{p+1} - b_{p+2} - c_{p+2}. \\
$$
\bigskip
So, using equation~\ref{long} and operating modulo $p$,
$$
\Lambda_{p-2, p-2} (\, s_1*Q(\, x, y \, )^{p-1} \,) =
 \begin{cases}
-xy(y+1)  \; & \text{if $\; p \equiv 1 \mod 6$;}\\
-xy(y+1) - 1  \; & \text{if $\; p \equiv -1 \mod 6.$}\\
\end{cases}
$$
\bigskip

The next state to appear is $\Lambda_{p-1, p-1} (\, s_1*Q(\, x, y \, )^{p-1} \,)$.
\begin{eqnarray*}
&\Lambda_{p-1, p-1}&( s_1*Q( x, y )^{p-1} ) = -a_{p-1} - (-a_p - b_p + 1 +xy + xy^2) +  (c_{p+1} + b_{p+1} ) \\
& & -   (c_{p+1} + 3 b_{p+1} + 2 a_{p+1} - 2) + 2(xy + xy^2) \\
&=& xy(y+1) - a_{p-1} + a_p + b_p - 2 b_{p+1} - 2 a_{p+1} + 1 \\
& & \\
&=&
\begin{cases}
xy(y+1) + 2 \; & \text{if $\; p \equiv 1 \mod 6$;}\\
xy(y+1) - 1 \; & \text{if $\; p \equiv -1 \mod 6$;}\\
\end{cases}
\end{eqnarray*}
using equation (\ref{long2}).
\bigskip
The transitions $\Lambda_{d, d}( s_2*Q( x, y )^{p-1} )$ all produce constant states except for 
$\Lambda_{0, 0}( s_2*Q( x, y )^{p-1} )$.
\begin{eqnarray*}
&\Lambda_{0, 0}&( s_2*Q( x, y )^{p-1} ) = a_{p-1} xy + 2 b_p xy + 2(-a_p - b_p + 1) xy + 2x^2y^2 + 2x^2y^3\\
&=& (\, a_{p-1} - 2 a_p + 2 \,) xy + 2x^2y^2 + 2x^2y^3\\
&=& s_2\\
& &
\end{eqnarray*}
using equation (\ref{anid}).
\bigskip
In order to complete the entries from tables \ref{results-1}, \ref{results-2} and \ref{results-3} it is enough to examine the transitons $\Lambda_{d, d}( xy(y+1)*Q( x, y )^{p-1} )$ (noting that $xy(y+1)$ is not actually a state). A similar calculation to the one for $\Lambda_{p-2, p-2}( s_1*Q( x, y )^{p-1} )$ can be used to show that
\begin{eqnarray*}
\Lambda_{p-3, p-3}( xy(y+1)*Q( x, y )^{p-1} ) &=& \frac{1}{2} (\, b_{p+2} - c_{p+2} \,) ;\\
& &\\
\Lambda_{p-3, p-3}( 1*Q( x, y )^{p-1} ) &=& \frac{1}{2} (\, c_{p+1} - b_{p+1} \,).\\
& &
\end{eqnarray*}
We then have
\begin{eqnarray*}
&\Lambda_{p-3, p-3}&( (\, -xy(y+1) - 1 \,)*Q( x, y )^{p-1} ) = \frac{1}{2} \biggr(\, (c_{p+2} - c_{p+1} ) - (b_{p+2} - b_{p+1} ) \,\biggr)\\
&=& \frac{1}{2} \biggr(\, (-a_p -2 b_p - c_p\, ) - (\, -a_p - b_p \,) \biggr)\\
&=& -\frac{1}{2} (c_p + b_p )\\
&\equiv & 0 \mod p\\
& &
\end{eqnarray*}
using equation (\ref{bcnid}).

We also have
\begin{eqnarray*}
&\Lambda_{p-1, p-1}&( (\, xy(y+1) \,)*Q( x, y )^{p-1} ) \\
&=& b_{p} + (\, -a_p -b_p + 1 \,) + xy + xy^2\\
&=& 1 - a_p + xy(y+1)\\
&=&
\begin{cases}
xy(y+1) \; & \text{if $\; p \equiv 1 \mod 6$;}\\
xy(y+1) + 1 \; & \text{if $\; p \equiv -1 \mod 6$.}\\
\end{cases}
\end{eqnarray*}

\section{Conclusions}
The values of $n$ mentioned in table~\ref{results} for which $M_n \equiv 0 \mod p$ can be immediately derived from an inspection of tables~\ref{results-1}, \ref{results-2} and \ref{results-3} and the associated state diagrams in figures~\ref{modpfigure1}, \ref{modpfigure2} and \ref{modpfigure3}.
\bigskip
A lower bound for the asymptotic density of the set $S_p(0)$ can then be derived from the following result from \cite{Burns:2016v0}
\bigskip
\begin{thm}
\label{asyden}
Let 
$$
S(\, q, r, s, t )\, = \{ (\, qi + r )\, q^{sj + t} : i, j \in \mathbb{N} \}
$$
and
$$
S^{'}(\, q, r, s, t )\, = \{ (\, qi + r )\, q^{sj + t} : i, j \in \mathbb{N}, j \geq 1 \}
$$
for integers $q, r, s, t \in \mathbb{Z} $ with $q, s > 0$, $t \geq 0$ and $0 \leq r < q$. Then the asymptotic density of the set $S$ is $(\, q^{t + 1 - s} (\, q^{s} - 1 )\, )\, ^{-1} $.  The asymptotic density of the set $S^{'}$ is \mbox{$(\, q^{t + 1} (\, q^{s} - 1 )\, )\, ^{-1} $}.
\end{thm}
\bigskip
From above we know that if $p \equiv 1 \mod p$ then $M_n \equiv 0 \mod p$ when n is in the forms
\begin{align*}
\mbox{$n = (pi + 1)p^k - 2 \;$ for $\; i \geq 0 \;$ and $\; k \geq 1$.} \\
\mbox{$n = (pi + p - 1)p^k - 1 \;$ for $\; i \geq 0 \;$ and $\; k \geq 1$.} \\ 
\\
\end{align*}

Each of the $2$ forms has asymptotic density $\frac{1}{p(p-1)}$. Therefore, when $p \equiv 1 \mod 6$ the asymptotic density of $S_p(0) \geq \frac{2}{p(p-1)}$. 

For $p \equiv -1 \mod 6$ there are $4$ forms of numbers to consider. These are 
\begin{align*}
\mbox{$n = (pi + 1)p^{2k} - 2 \; $ for $i \geq 0 \;$ and $\; k \geq 1$.}\\
\mbox{$n = (pi + p - 2)p^{2k+1} - 2 \; $ for $i \geq 0 \;$ and $\; k \geq 0$.} \\
\mbox{$n = (pi + 2)p^{2k+1} - 1 \;$ for $i \geq 0 \;$ and $\; k \geq 0$.} \\
\mbox{$n = (pi + p - 1)p^{2k} - 1 \; $ for $i \geq 0 \;$ and $\; k \geq 1$.} \\
\\
\end{align*}
The first and fourth forms have asymptotic density $(\, p(\, p^2 - 1 \,) \,)^{-1}$. 
The second and third forms have asymptotic density $(\, p^2 - 1 \,)^{-1}$. 
Therefore, when $p \equiv -1 \mod 6$ the asymptotic density of $S_p(0)$ is again $\geq \frac{2}{p(p-1)}$. 
\bigskip

Tables~\ref{results-1}, \ref{results-2} and \ref{results-3} and the state diagrams in Figures~\ref{modpfigure1}, \ref{modpfigure2} and \ref{modpfigure3} can be used to determine which numbers $n$ have $M_n \equiv x \mod p$ for other values of $x$. For example, figure \ref{modpfigure2} shows that if $p \equiv 1 \mod 6$ then \mbox{$M_n \equiv 1 \mod p \;$}  when the base $p$ representation of $n$ contains only $0$'s and $1$'s. Figure \ref{modpfigure1} shows that if $p \equiv 1 \mod 6$ then \mbox{$M_n \equiv 2 \mod p \;$}  when $n = p^k -1$ for some $k \in \mathbb{N}$.

\bigskip

As mentioned in the introduction results on forbidden residues of $M_n \mod p^k$ have been proved for some primes $p$ and some $k \geq 2$. In order to whether there are forbidden residues $\mod p$ itself the constant states of the automata would need to be examined. Since
$$
\Lambda_{d, d} ( 1 * Q( x, y )^{p-1} ) = c(p, d) :=\sum_{k \geq 0} \binom{p-1-k}{d-2k} \binom{d}{k} (-1)^d
$$
in order to show there are no forbidden residues $\mod p$ it is sufficient to show that the set
$$
\{ \, c(p, d):  \, 0 \leq d \leq p-1 \} 
$$
generates $( \frac{ \mathbb{Z}}{p \mathbb{Z}} )^{\times}$. 

\bigskip

As shown in \cite{1612.08146} the asymptotic density of $S_p(0)$ is actually $1$ for some primes, e.g. $p = 7, 17,19$. In these cases, there is a $d: 2 \leq d \leq p-2$ such that
$$
\Lambda_{d, d} ( 1 * Q( x, y )^{p-1} ) = 0.
$$
It then follows that
$$
\Lambda_{d, d} ( c * Q( x, y )^{p-1} ) = 0
$$
for all constants $c$. As a result, any $n$ which has a base-$p$ representation containing $2$ or more digits $d$ satisfies $M_n \equiv 0 \mod p$. The asymptotic density of this set is $1$. It would therefore be of interest to determine for which $p$ and $d$
$$
\sum_{k \geq 0} \binom{p-1-k}{d-2k} \binom{d}{k} \equiv 0 \mod p.
$$

\bigskip

\bibliographystyle{plain}
\begin{small}
\bibliography{ref}
\end{small}

\end{document}